\documentclass[12pt]{article}
\usepackage{geometry, amsmath, amsthm, nicefrac, mathrsfs,amsfonts,commath,color}                % See geometry.pdf to learn the layout options. There are lots.
\usepackage{graphicx}
\usepackage{amssymb}
\usepackage{epstopdf}
\usepackage{verbatim}
\def\to{\longrightarrow}

\def\QQ{{\mathbb Q}}

\def\ZZ{{\mathbb Z}}

\def\S{{\mathcal{S}}}

\newtheorem{theorem}{Theorem}
\newtheorem{lemma}[theorem]{Lemma}
\newtheorem{proposition}[theorem]{Proposition}

\newtheorem{rmk}[theorem]{Remark}

\renewcommand{\restriction}{\mathord{\upharpoonright}}

\title{$H^2(\mathbf{SL_3}(\mathbb{Z}[t]); \mathbb{Q}$) is infinite dimensional}
\author{Morgan Cesa and Brendan Kelly}
\begin{document}
\maketitle

\begin{abstract}
We prove that $H^2(\mathbf{SL_3}(\mathbb{Z}[t]); \mathbb{Q})$ is infinite dimensional. The proof follows an outline similar to recent results by Cobb, Kelly, and Wortman, using the Euclidean building for $\mathbf{SL_3}(\mathbb{Q}((t^{-1})))$ and a Morse function from Bux-K\"ohl-Witzel.
\end{abstract}

\section{Introduction}\label{intro}
Krsti\`c-McCool proved that $\mathbf{SL_3}(\mathbb{Z}[t])$ is not finitely presented \cite{KrsticMcCool1999}. In \cite{BuxMohammadiWortman2010}, Bux-Mohammadi-Wortman show that $\mathbf{SL_n}(\mathbb{Z}[t])$ is not $FP_{n-1}$. In general, a group $G$ being of type $FP_k$ implies that $H^k(G;M)$ must be finitely generated, {where $M$ is a $\mathbb{Z}G$-module}.  

In \cite{Wortman2013}, Wortman exhibits a finite index subgroup  $\Gamma \leqslant \mathbf{SL_n}(\mathbb{F}_q[t])$ such that $H^{n-1}(\Gamma; \mathbb{F}_p)$ is infinite dimensional. In \cite{Cobb2015}, Cobb shows that $H^2(\mathbf{SL_2}(\mathbb{Z}[t,t^{-1}]); \mathbb{Q})$ is infinite dimensional. In \cite{Kelly2013}, Kelly exhibits a finite index subgroup $\Gamma \leqslant \mathbf{B_n}(\mathbb{F}_q[t,t^{-1}])$ such that $H^2(\Gamma; \mathbb{F}_p)$ is infinite dimensional, where $\mathbf{B_n}(\mathbb{F}_p[t,t^{-1}])$ is the upper triangular subgroup of $\mathbf{SL_n}(\mathbb{F}_p[t,t^{-1}])$ and $p\neq2$.

 In this paper, we prove the following result:
 \begin{theorem} \label{theorem:main} $H^2(\mathbf{SL_3}(\ZZ[t]); \QQ)$ is infinite dimensional.
 \end{theorem}

We will let $\Gamma = \mathbf{SL_3}(\ZZ[t])$ and $G = \mathbf{SL_3}(\mathbb{Q}((t^{-1})))$. 

First, we will use ideas from Bux-K\"ohl-Witzel \cite{BuxKohlWitzel2013} to define an $\mathbf{SL_3}(\mathbb{Q}[t])$-invariant piecewise linear Morse function on the Euclidean building for $\mathbf{SL_3}(\mathbb{Q}((t^{-1})))$.  Then we will construct a $2$-connected $\Gamma$-complex $Y,$ which will be built from a connected subset of the Euclidean building by gluing cells as freely as possible until we arrive at a $2$-connected complex. We will show that $H^2(\Gamma \backslash Y; \QQ)$ is infinite dimensional by constructing infinite linearly independent families of $2$-cocycles and $2$-cycles that pair nontrivially. Finally, we will use the equivariant homology spectral sequence with $$E^2_{p, q} = H_p(\Gamma \backslash Y; \{H_q(\Gamma_\sigma; \QQ)\}) \Rightarrow H_{p+q}(\Gamma; \QQ)$$ to show that the infinite dimension of $H^2(\Gamma \backslash Y; \QQ)$ implies that $H^2(\Gamma; \QQ)$ is infinite dimensional.

The authors wish to thank their Ph.D. advisor, Kevin Wortman, for his valuable insights and detailed explanations of his results. Thanks also to Sarah Cobb for helpful conversations.

\section{Preliminaries}
Let $X$ be the Euclidean building for $\mathbf{SL_3}(\mathbb{Q}((t^{-1})))$. $X$ is a $2$-dimensional simplicial complex, with vertices corresponding to the homothety classes of 3-dimensional $\mathbb{Q}[[t^{-1}]]$-lattices (two lattices are in the same homothety class if one is a nonzero scalar multiple of the other) in $\mathbb{Q}((t^{-1}))^3$. 
A basis $\{v_1, v_2, v_3\}$ for $\mathbb{Q}((t^{-1}))^3$ gives rise to the $\mathbb{Q}[[t^{-1}]]$-lattice $$v_1\mathbb{Q}[[t^{-1}]] \oplus v_2 \mathbb{Q}[[t^{-1}]] \oplus v_3\mathbb{Q}[[t^{-1}]]$$
We will let $v_1\oplus v_2 \oplus v_3$ denote the lattice above.
 
Note that $\mathbf{SL_3}(\mathbb{Q}((t^{-1})))$ acts linearly on the vector space $\mathbb{Q}((t^{-1}))^3$, and therefore on $\mathbb{Q}[[t^{-1}]]$-lattices, and this gives an $\mathbf{SL_3}(\mathbb{Q}((t^{-1})))$-action on the vertices of $X$. Let $x_0$ represent the vertex corresponding to the equivalence class of the $\mathbb{Q}[[t^{-1}]]$-lattice generated by the standard basis, $e_1=(1, 0, 0 ), e_2=(0, 1, 0),$ and $e_3=(0, 0, 1)$. The $\mathbf{SL_3}(\mathbb{Q}((t^{-1})))$-stabilizer of $x_0$ is $\mathbf{SL_3}(\mathbb{Q}[[t^{-1}]])$. Let $\mathcal{A}_0$ represent the apartment of $X$ which is stabilized by the diagonal subgroup of $\mathbf{SL}_3(\mathbb{Q}((t^{-1})))$, and let $\mathcal{C}_0$ represent the chamber in $\mathcal{A}_0$ which contains $x_0$ and is stabilized by the subgroup of upper-triangular matrices in $\mathbf{SL_3}(\mathbb{Q}[[t^{-1}]])$. 
We will refer to $\mathcal{A}_0$ as the \emph{standard apartment}, $\mathcal{C}_0$ as the \emph{standard chamber}, and $x_0$ as the \emph{standard vertex}. 

The subgroup of permutation matrices (matrices with exactly one entry of $\pm1$ in each row and column, and all other entries 0) acts transitively on the 6 chambers in $\mathcal{A}_0$ which contain $x_0$. There are 6 sectors in $\mathcal{A}_0$ based at $x_0$, separated by the three walls in $\mathcal{A}_0$ which pass through $x_0$, and the permutation subgroup acts transitively on these sectors.  Let $\mathcal{S}_0$ be the sector which contains the standard chamber $\mathcal{C}_0$. 
$\mathcal{S}_0$ is a strict fundamental domain for the action of $\mathbf{SL_3}(\mathbb{Q}[t])$ on $X$ \cite{Soule1977}.  

Let $X_\Gamma = \Gamma \mathcal{S}_0$, and observe that $\mathcal{A}_0 \subset X_\Gamma$ because $\Gamma$ contains the permutation matrices in $\mathbf{SL_3}(\mathbb{Z})$ which act transitively on the sectors of $\mathcal{A}_0$ based at $x_0$.

\subsection{Cell Stabilizers}
In this section, we will discuss the $\Gamma$-stabilizers of cells in $\mathcal{S}_0$. For simplicity, we will let $\Gamma_\sigma = Stab_\Gamma(\sigma)$ and $G_\sigma = Stab_G(\sigma)$ for a cell $\sigma \subset X$.  (Recall that $G = \mathbf{SL_3}(\QQ((t^{-1})))$ and $\Gamma = \mathbf{SL_3}(\ZZ[t])$.)

\begin{lemma}\label{lemma:stabforms}   If $x$ is a vertex in $\mathcal{S}_0$, then $\Gamma_x$ has one of the following forms, where $u, v, w \in \mathbb{Z}[t]$, $a, b, c, d \in \mathbb{Z}$ such that $|ad-bc| = 1$, and $k$ and $m$ are nonnegative integers which depend on $x$.\begin{enumerate}
\item If $x_0$ is the standard vertex of $X$, then $\Gamma_{x_0} = \mathbf{SL_3}(\mathbb{Z})$.
\item If $x$ is a vertex in the interior of $\mathcal{S}_0$, then
$$\Gamma_x = \left\{\left(\begin{array}{ccc}\pm 1 & u & w \\0 & \pm 1 & v \\0 & 0 & \pm 1\end{array}\right)\middle| deg(u)\leq k, deg(v) \leq m, deg(w)\leq m+k\right\}$$

\item If $x$ is a vertex in $\partial \mathcal{S}_0$, and $x \neq x_0$, then $\Gamma_x$ has one of the following forms:

$$\Gamma_x = \left\{\left(\begin{array}{ccc}a & b & w \\c & d & v \\0 & 0 & \pm 1\end{array}\right)\middle|  deg(w), deg(v) \leq k \right\}$$

$$\Gamma_x = \left\{\left(\begin{array}{ccc}\pm1 & u & w \\0 & a & b \\0 & c & d\end{array}\right)\middle|  deg(u), deg(w) \leq k \right\}$$

%\item If $x = y_n$, then
%$$\Gamma_{y_n}= \left\{\left(\begin{array}{ccc}\pm1 & u & w \\0 & \pm1 & v \\0 & 0 & \pm1\end{array}\right)\middle| deg(u), deg(v) \leq n, deg(w)\leq 2n+1\right\}$$
\end{enumerate}

\end{lemma}
\begin{proof}
First, observe that $G_{x_0} = \mathbf{SL_3}(\mathbb{Q}[[t^{-1}]])$, and therefore $$\Gamma_{x_0} = G_{x_0} \cap \Gamma = \mathbf{SL_3}(\mathbb{Z})$$
Any vertex $x$ in $\mathcal{S}_0$ corresponds to a $\mathbb{Q}[[t^{-1}]]$-lattice of the form 
$$t^i e_1 \oplus t^je_2 \oplus e_3$$ 
for nonnegative integers $j \leq i$, where $\{e_1, e_2, e_3\}$ is the standard basis for $\mathbb{Q}((t^{-1}))^3$. Any vertex in $\partial \mathcal{S}_0$ corresponds to a lattice with either $j = 0$ or $i = j$. Letting $$g =\left(\begin{array}{ccc}t^i & 0 & 0 \\0 & t^j & 0 \\0 & 0 & 1\end{array}\right)$$ we have $$g (e_1 \oplus e_2 \oplus e_3) = t^i e_1 \oplus t^j e_2 \oplus e_3$$
Therefore, $\Gamma_x = (g \mathbf{SL_3}(\mathbb{Q}[[t^{-1}]])g^{-1})\cap \Gamma$.  
Computing $gAg^{-1}$ for an arbitrary matrix $A \in \mathbf{SL_3}(\mathbb{Q}[[t^{-1}]])$ gives
$$gAg^{-1}=g\left(\begin{array}{ccc}a_{11} & a_{12} & a_{13} \\a_{21} & a_{22} & a_{23} \\a_{31} & a_{32} & a_{33}\end{array}\right)g^{-1} = \left(\begin{array}{rrr}a_{11} & t^{i-j}a_{1 2} & t^ia_{13} \\t^{j-i}a_{21} & a_{22} & t^ja_{23} \\t^{-i}a_{31} & t^{-j}a_{32} & a_{33}\end{array}\right)$$
where $a_{ij} \in \mathbb{Q}[[t^{-1}]]$. 
If $gAg^{-1} \in \Gamma$, then we obtain the following form for $gAg^{-1}$:
$$\left(\begin{array}{lll}deg = 0 & deg \leq (i-j) & deg \leq i \\deg \leq j-i & deg = 0 & deg \leq j \\deg \leq -i & deg \leq -j & deg =0\end{array}\right)$$
If $x$ is in the interior of $\mathcal{S}_0$, then $i> j > 0$ and we take $k = i-j$ and $m = j$.

If $x$ is in the boundary of $\mathcal{S}_0$, then either $j = 0$ or $i = j$. If $j = i = 0$, then $x = x_0$, so we may assume $i\neq 0$ . In either case ($j=0$ or $i = j$) we take $k = i$.  Depending on whether or not $j=0$, we obtain one of the two forms for $\Gamma_x$ stated in the lemma.

%Finally, we must consider the case where $x = y_n$. In this case, $y_n$ is the barycenter of the edge spanned by the vertices $x$ and $x'$, corresponding to $t^{2n+1} e_1\oplus t^ne_2\oplus e_3$ and $t^{2n+1} e_1 \oplus t^{n+1} e_2 \oplus e_3$, respectively. To complete the proof, we observe that $\Gamma_{y_n} = \Gamma_x \cap \Gamma_{x'}$.

\end{proof}

\begin{lemma} For $\sigma$ a subcell of $\mathcal{C}_0$, $\Gamma_\sigma$ is of type $F_1$.
\end{lemma}
A much stronger result is proved in \cite{BuxMohammadiWortman2010}, where it is shown that if $\sigma$ is any cell in $X$, then $\Gamma_\sigma$ is of type $F_\infty$.  However, we only make use of the specific case above, and provide a short proof here:
\begin{proof}
First, recall that a group is type $F_1$ if and only if it is finitely generated. First, suppose $\sigma$ is a $0$-cell. It is easy to see that $\Gamma_\sigma$ is finitely generated by Lemma \ref{lemma:stabforms}.

Let $e_{ij}(a)$ represent the elementary matrix with $a$ in the $ij^\text{th}$ entry, 1's on the diagonal and 0's elsewhere. 

Suppose $\sigma$ is a $1$-cell in $\mathcal{C}_0$.  If $\sigma$ contains $x_0$, then $\Gamma_\sigma$ is a maximal parabolic subgroup of $\mathbf{SL}_3(\mathbb{Z})$ and is therefore finitely generated.
If $\sigma$ does not contain $x_0$, then $\Gamma_\sigma$ is upper-triangular and generated by $e_{12}(1), e_{23}(1), e_{13}(1), e_{13}(t)$, and the finite diagonal subgroup of $\mathbf{SL}_3(\mathbb{Z})$.

Finally, suppose $\sigma = \mathcal{C}_0$. In this case $\Gamma_\sigma$ is the upper-triangular subgroup of $\mathbf{SL_3}(\mathbb{Z})$ and it is easy to see that this group is finitely generated.

\end{proof}

\subsection{Morse Function}
If $Z$ is a CW-complex, let $Z^{(i)}$ denote the $i$-skeleton of $Z$.  

A function $h:X\to\mathbb{R}$ is a \emph{piecewise linear Morse function} (or \emph{Morse function}) if $h$ restricts to an affine (height) function on every simplex, $h(X^{(0)})$ is discrete, and $h$ is not constant on any simplex of dimension at least 1. Our goal in this section will be to define a $\Gamma$-invariant Morse function on $X_\Gamma$, and an $\mathbf{SL_3}(\QQ[t])$-invariant Morse function on $X$. Since the standard sector, $\mathcal{S}_0$, is a strict fundamental domain for $\Gamma$ acting on $X_\Gamma$ (respectively, for $\mathbf{SL_3}(\mathbb{Q}[t])$ acting on $X$), any Morse function on $\mathcal{S}_0$ can be extended to a $\Gamma$-invariant Morse function on $X_\Gamma$ (respectively, to an $\mathbf{SL_3}(\mathbb{Q}[t])$-invariant Morse function on $X$).

The Morse function we define on $X$ is essentially the same one defined by Bux-K\"ohl-Witzel \cite{BuxKohlWitzel2013}.  We will make this statement more precise in Remark \ref{rmk:BKW}.

Define a function $\hat h$ on $\mathcal{S}_0 ^{(0)}$ by $\hat h(x) = d(x_0, x)$, where $d$ is the Euclidean metric on $\mathcal{A}_0$. A first attempt at extending $\hat h$ to $\mathcal{S}_0$ would be to extend using barycentric coordinates on each simplex.  However, there is a sequence of edges in the middle of the sector which are flat with respect to this extension. 
\begin{figure}[h]
	\centering
	\includegraphics[height=2.5in]{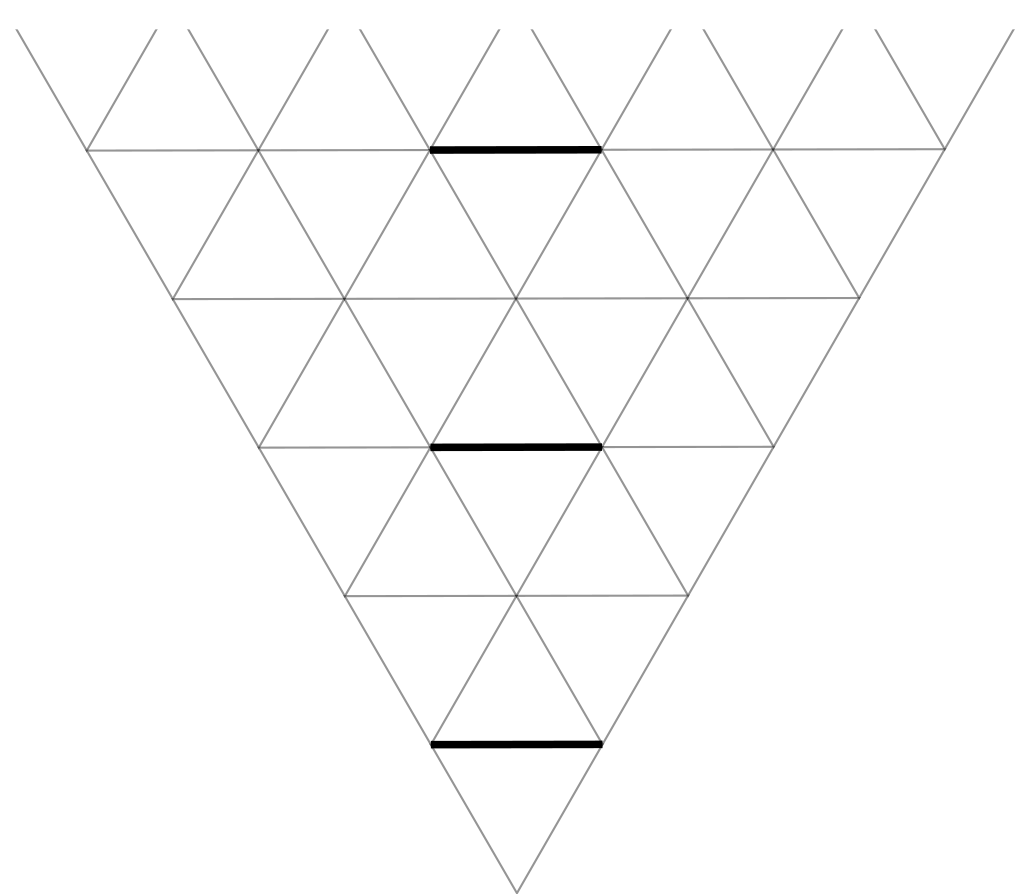}
	\caption{The sector $\mathcal{S}_0$, with the edges which are flat under $\hat h$ highlighted.}
\end{figure}
We denote this sequence by $\{\eta_n\}_{n \in \mathbb{N}}$.  Specifically, $\eta_n$ is the edge spanned by the vertices 
$t^{2n+1}e_1 \oplus t^{n} e_2 \oplus e_3$ and $t^{2n+1} e_2 \oplus t^{n+1}e_2 \oplus e_3$.

For each $n$, $\eta_n$ is contained in two chambers of $\mathcal{S}_0$. Let $\mathcal{C}_n^\uparrow$ be the chamber in $\mathcal{S}_0$ which is above $\eta_n$ (more precisely, the chamber with $\hat h(v) > \hat h(\eta_n^{(0)})$ for the vertex $v$ which is not in $\eta_n$), and $\mathcal{C}_n^\downarrow$ the chamber below $\eta_n$. Let $\mathring X$ denote the barycentric subdivision of $X$, and similarly let $\mathring{\sigma}$ denote the barycenter of a cell $\sigma \subset X$. We will extend $\hat h$ to cells in $\mathring{\mathcal{S}}_0^{(0)}$ which do not intersect $\{\eta_n\}_{n \in \mathbb{N}}$ using barycentric coordinates, then choose $\hat h(\mathring \eta_n)$ such that $$\hat h(\partial \eta_{n+1}) >  \hat h(\mathring \eta_n) > \hat h(\partial \eta_n)$$
\begin{figure}
	\centering
	\includegraphics[height=2.5in]{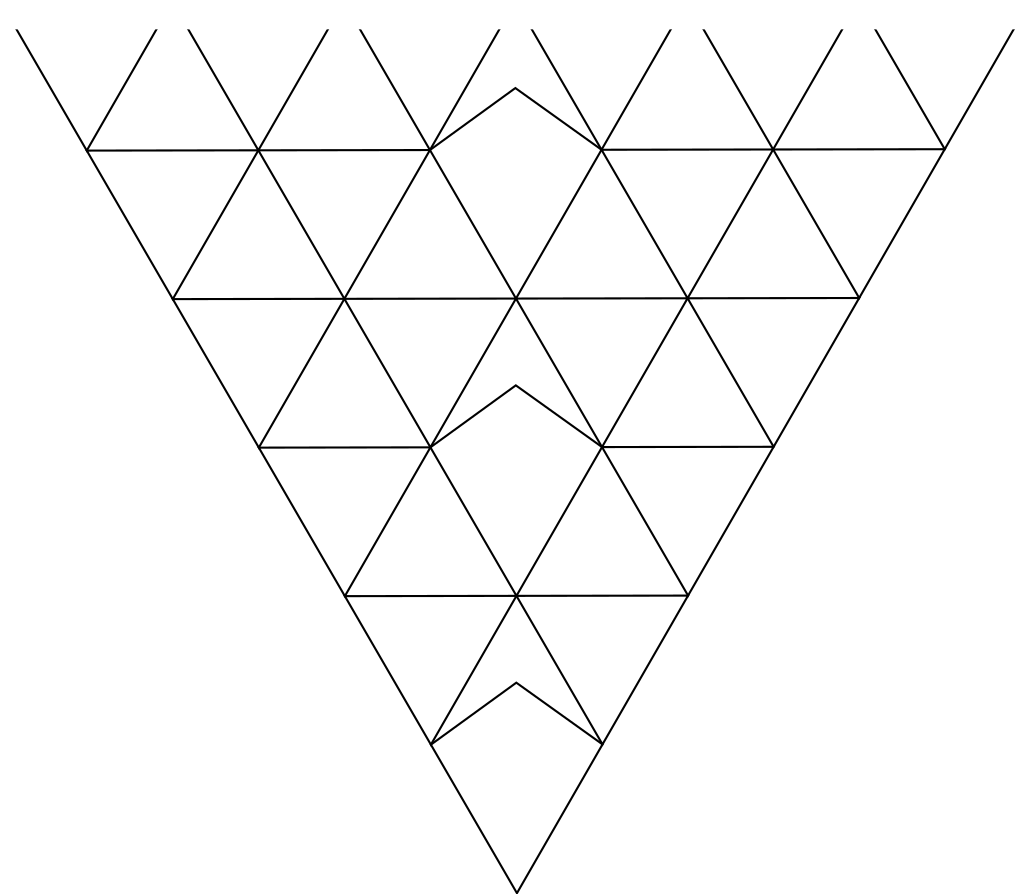}
	\caption{The sector $\mathcal{S}_0$ with $\hat{h}(\mathring{\eta}_n)$ redefined.}
\end{figure}
Finally, extend $\hat h$ to cells which intersect $\{\eta_n\}_{n\in\mathbb{N}}$ using barycentric coordinates.

Note that $\hat h$ is discrete on the vertices of $\mathring{\mathcal{S}}_0$ and bounded below by 0, so there is a function $h: \mathring{\mathcal{S}}_0^{(0)} \to \mathbb{Z}$ such that $h$ and $\hat h$ induce the same ordering on $\mathring{\mathcal{S}}_0^{(0)}$ and $h(x_0) = 0$. Extending $h$ to the $1$- and $2$-cells of $\mathring{\mathcal{S}}_0$ by using barycentric coordinates, then $\Gamma$-invariantly to $\mathring X_\Gamma,$ we obtain a $\Gamma$-invariant function on $\mathring X_\Gamma$. We may also extend $h$ to an $\mathbf{SL_3}(\mathbb{Q}[t])$-invariant function, $\bar{h}$, on $\mathring X$, because $\mathcal{S}_0$ is a strict fundamental domain for the action of $\mathbf{SL_3}(\mathbb{Q}[t])$ on $X$.

Let $y_n = \mathring \eta_n$.

\begin{lemma}
$h$ and $\bar h$ are piecewise linear Morse functions.
\end{lemma}

\begin{proof}
It suffices to show that $h\restriction_{\mathring{\mathcal{S}}_0} = \bar h \restriction_{\mathring{\mathcal{S}}_0}$ is Morse, since $h$ and $\bar h$ are respectively $\Gamma$- and $\mathbf{SL_3}(\mathbb{Q}[t])$-invariant and $\mathring{\mathcal{S}}_0$ is a strict fundamental domain for the respective group actions on $\mathring X_\Gamma$ and $\mathring X$.
By construction, $h(\mathring{\mathcal{S}}^{(0)})$ is discrete in $\mathbb{R}$. Since $ h$ is defined on $1$- and $2$-simplices by using barycentric coordinates, $h$ restricts to a height function on simplices. 

Let $\sigma \in \mathring{\mathcal{S}}$ be a cell. We must show that if $h$ is constant on $\sigma$ then $\sigma$ is a vertex. By construction, $h\restriction_\sigma$ is constant if and only if $h\restriction_{\sigma^{(0)}}$ is constant. Therefore, it suffices to show that $h$ is not constant on any $1$-cells of $\mathcal{S}_0$.

Suppose $\sigma$ is a $1$-cell.  If $\sigma$ does not contain $y_n$, then $h\restriction_{\sigma}$ is not constant because $\hat h$ is not constant on any $2$-cells, or on $1$-cells which do not contain $y_n$. If $\sigma$ contains $y_n$, then $h\restriction_\sigma$ is not constant by our choice of $h(y_n)$.
\end{proof}

\begin{rmk} \label{rmk:BKW} We note that $\bar{h}$ is essentially the same as the Morse function defined in \cite{BuxKohlWitzel2013}. The proof of Bux-K\"{o}hl-Witzel requires only the input of a uniform, $\mathbf{SL_3}(\QQ((t^{-1})))$-invariant reduction datum for $X$. In the most general context of  Bux-K\"{o}hl-Witzel, this reduction datum is supplied for arithmetic groups over function fields by Harder's reduction theory. However, in the specific case of $\mathbf{SL_n}(\mathbb{F}_p[t])$, there exists a reduction theory that is more precise than Harder's. Namely, the action of $\mathbf{SL_n}(\mathbb{F}_p[t])$ on its Euclidean building admits a strict fundamental domain for its action on its Euclidean building, and this fundamental domain is exactly a sector. A proof of this last statement is given by Soul\'{e}  in \cite{Soule1977}. Notice that in the result of Soul\'{e}, that the fields of coefficients for the polynomial rings are arbitrary, and thus the same statement applies equally as well to $\mathbf{SL_3}(\QQ((t^{-1})))$, thus supplying a uniform, $\mathbf{SL_3}(\QQ((t^{-1})))$-invariant reduction datum for $X$, and now the proof of Bux-K\"{o}hl-Witzel applies without modification.
\end{rmk}

\begin{lemma}\label{lemma:stabyn}
If $y_n$ is defined as above, then
$$\Gamma_{y_n}= \left\{\left(\begin{array}{ccc}\pm1 & u & w \\0 & \pm1 & v \\0 & 0 & \pm1\end{array}\right)\middle| deg(u), deg(v) \leq n, deg(w)\leq 2n+1\right\}$$
\end{lemma}
\begin{proof}
Recall that $y_n$ is the barycenter of the edge spanned by the vertices $x$ and $x'$, corresponding to the lattices $t^{2n+1} e_1\oplus t^ne_2\oplus e_3$ and $t^{2n+1} e_1 \oplus t^{n+1} e_2 \oplus e_3$, respectively. 

By Lemma \ref{lemma:stabforms}, 
$$\Gamma_x = \left\{\left(\begin{array}{ccc}\pm 1 & u & w \\0 & \pm 1 & v \\0 & 0 & \pm 1\end{array}\right)\middle| deg(u)\leq n+1, deg(v) \leq n, deg(w)\leq 2n+1\right\}$$
$$\Gamma_{x'} = \left\{\left(\begin{array}{ccc}\pm 1 & u & w \\0 & \pm 1 & v \\0 & 0 & \pm 1\end{array}\right)\middle| deg(u)\leq n, deg(v) \leq n+1, deg(w)\leq 2n+1\right\}$$
To complete the proof, we observe that $\Gamma_{y_n} = \Gamma_x \cap \Gamma_{x'}$.
\end{proof}

Let $\mathcal{S}'_0 = \mathcal{S}_0 \cup \{y_n\}_{n \in \mathbb{N}}$ be the standard sector modified to include the vertices $y_n$. Note that both $h$ and $\bar h$ restrict to Morse functions on $\mathcal{S}'_0$. 
Let $\mathcal{A}'_0$ be the standard sector modified to include the vertices $y_n$ and their images in each sector based at $x_0$, and let $X' = \mathbf{SL_3}(\mathbb{Q}[t]) \mathcal{S}'_0$ and $X'_\Gamma = \mathbf{SL_3}(\mathbb{Z}[t])\mathcal{S}'_0$ be the analogously modified versions of $X$ and $X_\Gamma$, respectively. Rather than using the barycentric subdivisions $\mathring X$, $\mathring X_\Gamma$, and $\mathring{\mathcal{S}}_0$, we will use $X'$, $X'_\Gamma$, and $\mathcal{S}'_0$. Note that $h$ restricts to a Morse function on $X'_\Gamma$ and $\bar h$ restricts to a Morse function on $X'$. We will abuse notation and use $h$ and $\bar h$ to denote the restricted Morse functions on $X'_\Gamma$ and $X'$.

\subsection{The descending star and descending link}

The \emph{star} of a vertex in a CW-complex is the union of all cells which contain that vertex. We denote the of a vertex $z$ in a CW-complex $Z$ by $St(z, Z)$. If $h$ is a piecewise linear Morse function on $Z$, then the \emph{descending star} of $z$, denoted $St^\downarrow(z, Z)$, is the subset of $St(z, Z)$ which consists of cells on which $h$ has a unique maximum at $z$: 
$$St^{\downarrow}(z, Z) = \left\{\text{cells }\sigma \in St(z, Z) \middle| h(v) < h(z) \text{ for every vertex } v \in \sigma-\{z\}\right\}$$

The \emph{link} of a vertex $z$ in a CW-complex $Z$ is the set of faces of cells in $St(z, Z)$ which have codimension 1 and do not contain $z$. We denote the link of $z$ in $Z$ by $Lk(z, Z)$. If $\sigma$ is a cell in $St(z, Z)$ we will use $\bar{\sigma}$ to denote the faces of $\sigma$ which are in $Lk(z, Z)$. The \emph{descending link} of $z$ is then
$$Lk^{\downarrow}(z, Z) = \left\{\text{cells }\bar{\sigma} \in Lk(z, Z) \middle| h(v) < h(z) \text{ for every vertex } v \in \bar{\sigma}\right\} = St^\downarrow(z, Z) \cap Lk(z, Z)$$

For simplicity of notation, when $Z$ is $X'$, we will suppress the simplicial complex and write $Lk^\downarrow(x)$ for $Lk^\downarrow(x, X')$ and $St^\downarrow(x)$ for $St^\downarrow(x, X')$.

\begin{lemma}\label{lemma:descendinglink}
$Lk^{\downarrow}(x)$ is connected for all $x \in X'$.
\end{lemma}

First we will prove a simpler lemma:
\begin{lemma} $Lk^{\downarrow}(x) \cap \mathcal{A}'_0$ is connected for all $x \in \mathcal{A}'_0 - \{x_0\}$ and consists of either 1 or 2 edges of $Lk(x)$. 
\end{lemma}
\begin{proof} We may assume $x \in \mathcal{S}'_0$. When $x \neq y_n$, this lemma is a consequence of Euclidean geometry and the fact that the angle spanned by the cells of $St^{\downarrow}(x) \cap \mathcal{A}'_0$ is strictly less than $\pi$. If $St(x)$ does not contain $y_n$, then $St(x, X') = St(x, X)$. Since the chambers of $X$ are equilateral triangles, with angles measuring $\frac{\pi}{3}$, there can be at most 2 chambers in $St^{\downarrow}(x) \cap \mathcal{A}'_0$.  If there are exactly 2 chambers in $St^\downarrow(x) \cap \mathcal{A}'_0$, they must share an edge and therefore $Lk^{\downarrow}(x) \cap \mathcal{A}'_0$ is connected. If $St(x)$ contains $y_n$, then there is at most 1 chamber in $Lk^\downarrow(x) \cap \mathcal{A}'_0$.
%\begin{figure}
%	\centering
%	\includegraphics[height=2.5in]{descending_links.png}
%	\caption{Several points of $\mathcal{S}'_0$ are colored, with descending links highlighted to match.}
%\end{figure}
When $x = y_n$, note that there are two cells in $\mathcal{S}'_0$ which contain $y_n$. Exactly one of these is in $St^\downarrow(y_n)$, and we will denote it by $\mathcal{C}^\downarrow_n$. 
Note that $\mathcal{C}^\downarrow_n$ is not a simplex, and so $Lk^\downarrow(y_n)\cap \mathcal{A}'_0$ consists of two edges. We will denote the union of these two edges by $\bar{\mathcal{C}}^\downarrow_n$.
\end{proof}

\begin{proof}[Proof of Lemma \ref{lemma:descendinglink}]
%{\color{blue} If $x$ is in the interior of $\mathcal{S}_0$, there is one chamber $\mathcal{C}_x$ of $\mathcal{S}_0$ which contains $x$ and ``faces towards" $x_0$, namely the chamber for which the sector of $\mathcal{A}_0$ based at $x_0$ and containing that chamber also contains $x_0$. We claim that $\Gamma_x$ is generated by elements which fix at least one vertex of $\mathcal{C}_x \cap Lk^{\downarrow}(x)$.
%If $x$ is in the boundary of $\mathcal{S}_0$, we can use an element of $\mathbf{SL_3}(\mathbb{Z})$ to obtain a chamber $\mathcal{C}'_x$ in $\mathcal{A}_0$ which is adjacent to $\mathcal{C}_x$.  In this case $\Gamma_x$ is generated by elements which fix at least one vertex of either $\mathcal{C}_x \cap Lk^{\downarrow}(x)$ or $\mathcal{C}'_x\cap Lk^{\downarrow}(x)$.}

Let $\Gamma^\mathbb{Q} = \mathbf{SL_3}(\mathbb{Q}[t])$. It suffices to show that $\Gamma^\mathbb{Q}_x$ is generated by elements which fix at least one vertex in $Lk^{\downarrow}(x)\cap \mathcal{A}'_0$. Since the diagonal subgroup of $\Gamma^\mathbb{Q}$ acts trivially on $X$, we may ignore these generators of $\Gamma^\mathbb{Q}_x$. We may assume that $x \in \mathcal{S}'_0$.

First, assume $x = y_n$. Then $Lk^{\downarrow}(y_n)$ is $\bar{\mathcal{C}}^\downarrow_n = \mathcal{C}^\downarrow_n \cap Lk(y_n)$. By Lemma \ref{lemma:stabyn}, $\Gamma^\mathbb{Q}_{y_n}$ has the following form:
$$\left\{\left(\begin{array}{ccc}\pm1 & u & w \\0 & \pm1 & v\\0 & 0 & \pm1\end{array}\right)\middle| u,v,w \in \mathbb{Q}[t], deg(u), deg(v) \leq n, deg(w)\leq 2n+1\right\}$$

Note that $\Gamma^\mathbb{Q}_{y_n}$ is generated by diagonal matrices, which act trivially on $Lk^{\downarrow}(y_n)$, and elementary matrices $e_{12}(u), e_{23}(v),$ and $e_{13}(w)$, where $u, v, w \in \mathbb{Q}[t]$ such that $deg(u)\leq n$, $deg(v) \leq n$, or $deg(w) = 2n+1$.   Generators with $u$ or $v$ nonzero are in the stabilizer of at least one vertex adjacent to $y_n$. The two corresponding root subgroups fix families of walls in $\mathcal{A}_0$ which are parallel to the walls containing the boundary of $\mathcal{S}_0$.  Elements of these root subgroups which stabilize $y_n$ must also fix $\mathcal{C}_n^\downarrow$. Generators with $w$ nonzero stabilize the edge $\eta_n$ of which $y_n$ is the barycenter, and hence these elements stabilize the two vertices of $\bar{\mathcal{C}}_n^\downarrow$ which are adjacent to $y_n$.

Next, consider the case when $x$ is in the interior of $\mathcal{S}_0$.  Note that $x$ and $x_0$ are not in a common wall of $\mathcal{A}_0$, so there is a unique sector $\mathcal{S}_x$ of $\mathcal{A}_0$ based at $x$ which contains $x_0$. This sector intersects $Lk^\downarrow(x)$ in one edge, $\bar{\mathcal{C}}_x$. Note that $\Gamma^\mathbb{Q}_x$ is generated by diagonal matrices, and elementary matrices of the form $e_{12}(u)$ or $e_{23}(v)$, where $u, v \in \mathbb{Q}[t]$ (elementary matrices in $\Gamma^\mathbb{Q}_x$ of the form $e_{13}(w)$ are commutators of elementary matrices of the other two forms in $\Gamma^\mathbb{Q}_x$).  Each of these groups fixes a wall of $\mathcal{S}_x$, and therefore fixes at least one vertex of $\bar{\mathcal{C}}_x$.

Now suppose $x$ is in the boundary of $\mathcal{S}_0$. There is a unique sector $\mathcal{S}_x$ based at $x$ which contains $x_0$ and intersects the interior of $\mathcal{S}_0$. $\mathcal{S}_x$ intersects $Lk^\downarrow(x)$ in an edge, $\bar{\mathcal{C}}_x$. There is a second sector based at $x$, $\mathcal{S}_x'$, which contains $x_0$ but is disjoint from the interior of $\mathcal{S}_0$, and there is some $g \in \Gamma^\mathbb{Q}_x \cap \Gamma^\mathbb{Q}_{x_0}$ such that $\gamma \mathcal{S}_x = \mathcal{S}_x'$. Elements of $\Gamma^\mathbb{Q}_x$ which are also in $\Gamma^\mathbb{Q}_{x_0}$ fix the wall of $\mathcal{A}_0$ which contains $x$ and $x_0$ (and therefore, fix a vertex of $\bar{\mathcal{C}}_x$). Elements of $\Gamma^\mathbb{Q}_x$ which have 0 in the upper right corner (i.e. those $\gamma$ such that $\gamma_{13} = 0$) fix a vertex of $g\bar{\mathcal{C}}_x$. These two types of elements, along with diagonal matrices, generate $\Gamma^\mathbb{Q}_x$.

\end{proof}

\section{Construction of a $2$-connected $\Gamma$-complex}

In this section we will prove the following proposition:
\begin{proposition}\label{prop:Y} There is a $2$-connected $\Gamma$-complex $Y$ with a $\Gamma$-equivariant map $\psi:~Y\to X$ such that $\psi(Y^{(1)})$ has bounded height with respect to $h$.  Furthermore, there are only finitely many $\Gamma$-orbits in $Y$ of cells with nontrivial stabilizers, and all nontrivial cell stabilizers are of type $F_1$.\end{proposition}

\subsection{An $\mathbf{SL_3}(\QQ[t])$-invariant subspace of the building}

\begin{proposition} \label{boundedheight} There is $\Gamma^\QQ$-invariant, connected subcomplex $Z\subseteq X$ whose distance from a single $\Gamma^\QQ$-orbit in $X$ is bounded. 
\end{proposition}

\begin{proof} This proposition is essentially proved by Bux-K\"{o}hl-Witzel in \cite{BuxKohlWitzel2013}. Indeed, if $\mathbb{Q}$ is replaced by $\mathbb{F}_p$ in the above proposition, then the proposition is proved in Section 10 of \cite{BuxKohlWitzel2013}, and it is the means by which it is shown that $\mathbf{SL_n}(\mathbb{F}_p[t])$ is of type $F_{n-2}$. Furthermore, replacing $\mathbb{F}_p$ by $\mathbb{Q}$ makes no changes in their proof. 
\end{proof}

\begin{proof}[Proof of Proposition \ref{prop:Y}]
Let $\mathcal{C}_0$ be the standard chamber, and let $Y_0 = \Gamma \cdot \mathcal{C}_0$. $Y_0$ is connected by Suslin's theorem, which states that $\Gamma$ is (finitely) generated by matrices which fix at least one vertex of $\mathcal{C}_0$. For any cell $\sigma \subset Y_0$, $\Gamma_\sigma$ is a conjugate of $\Gamma_{\sigma_0}$ for some subcell $\sigma_0 \subset \mathcal{C}_0$ and thus $\Gamma_\sigma$ is of type $F_1$.

If $Y_0$ is not simply connected, there is a map $f:S^1 \to Y_0$ with noncontractible image. For each $\gamma \in \Gamma$, attach a 2-cell $\Delta^2_\gamma$ to $Y_0$ by identifying the boundary of $\Delta^2_\gamma$ with $\gamma f(S^1)$. Note that $\Gamma$ acts on these new 2-cells by permuting the indices.
%$$Y_0' \bigsqcup_{\gamma \in \Gamma} \Delta^2_\gamma /\sim\ , \ \ \ \partial \Delta^2_\gamma \sim \gamma f(S^1)$$

If the resulting space is not simply connected, repeat the above process with any remaining nontrivial $1$-spheres until the resulting space is simply connected. Call this space $Y_1$. Define a $\Gamma$-equivariant map $\psi: Y_1 \to X$ by mapping $\Delta^2_\gamma$ to the unique filling disk in $X$ of $\gamma f(S^1)$. If $\sigma$ is a cell in $Y_1 - Y_0$, then $\Gamma_\sigma= \{1\}$ by construction.

If $Y_1$ is not $2$-connected, there is a map $f:S^2 \to Y_1$ with noncontractible image.  Duplicate the process above, attaching a family of $3$-disks to $Y_1$ along the $\Gamma$ orbit of $f(S^2)$, and repeating the process if necessary until the resulting space is $2$-connected.  Call this space $Y$. Again, any cell in $Y - Y_1$ has trivial stabilizer. $X$ is $2$-dimensional and aspherical, and there are no nontrivial $2$-spheres in $\psi(Y_1)$.  Therefore, we may extend $\psi$ by mapping each $3$-disk continuously to the image of its boundary in $X$.

By construction, $Y^{(1)}$ has bounded height under $h$. Any cell in $Y$ with nontrivial stabilizer is contained in $Y_0$, and $Y_0$ is in the $\Gamma$-orbit of $\mathcal{C}_0$ and we have shown that the stabilizers of cells in this orbit are type $F_1$.  

\end{proof}

\begin{rmk}Note that the application of Suslin's theorem above (to show that $Y_0$ is connected) is not necessary, although it is convenient. If $Y_0$ were not connected, one could construct a connected complex in the following way:  let $p:[0,1]\to X$ be a path in $X$ between two components.  By Proposition \ref{boundedheight}, $p$ can be chosen so that its height under $h$ is bounded, regardless of the choice of components. For each $\gamma \in \Gamma$, attach a 1-cell $p_\gamma$ to $Y_0$ by identifying the endpoints of $p_\gamma$ with the endpoints of $\gamma(p)$. Note that $\Gamma$ acts on $\{p_\gamma\}$ by permuting the indices.
%$$Y_0 \bigsqcup_{\gamma \in \Gamma} p_\gamma /\sim\ , \ \ \  p_{\gamma}(0) \sim \gamma p(0), p_\gamma(1) \sim\gamma p(1)$$
If the resulting space is not connected, repeat the process with any remaining connected components.  Call the connected space $Y_0'$, and note that there is a $\Gamma$ equivariant map $\psi:Y_0'\to X$ such that $\bar h\circ \psi(Y_0')$ is bounded.  For any cell $\sigma \subset Y_0'$, $\Gamma_\sigma = \{1\}$ if $\sigma \notin Y_0$. The above proof of the existence of the complex $Y$ can be adapted to a more general setting.
\end{rmk}

\section{Cocycles and Cycles in $\Gamma \backslash Y$}
In this section, we prove the following:

\begin{proposition}\label{prop:infdim} $H_2(\Gamma \backslash Y; \QQ)$ is infinite dimensional.
\end{proposition}
We will prove this proposition by defining an infinite family of independent cocycles $\{\Phi_n\}_{n \in \mathbb{N}} \subseteq H^2(\Gamma \backslash Y; \QQ)$.  Then we will exhibit an infinite family of cycles in $H_2(\Gamma \backslash Y; \QQ)$, and use the cocycles $\Phi_n$ to show that these cycles are independent.

In order to define $\Phi_n$, we will first discuss a quotient of $X$, and define a family $\varphi_n$ of local cocycles on that quotient, then use $\varphi_n$ to define the cocycles $\Phi_n$ on $\Gamma \backslash Y$. 

\subsection{Congruence Subgroups of $\mathbf{SL_3}(\mathbb{Q}[t])$}

In this subsection, we will make a brief digression to discuss congruence subgroups of $\mathbf{SL_3}(\mathbb{Q}[t])$, in order to define a local cocycle in the next section.

There is a sequence of \emph{congruence subgroups} of $\mathbf{SL_3}(\mathbb{Q}[t])$ given by 
$$\mathbf{SL_3}(\mathbb{Q}[t], (t^n)) = \ker(\mathbf{SL_3}(\mathbb{Q}[t]) \rightarrow \mathbf{SL_3}(\mathbb{Q}[t]/(t^n)))$$

Let $U$ denote the upper-triangular subgroup of $\mathbf{SL_3}(\mathbb{Q}[t])$ and let $U_{n}$ denote the upper-triangular subgroup $U \cap \mathbf{SL_3}(\mathbb{Q}[t], (t^{n+1}))$. (Note that $U_n \unlhd U$, and $U_{n}\backslash U$ can be identified with the upper-triangular subgroup of $\mathbf{SL_3}(\mathbb{Q}[t]/(t^{n+1}))$.)

Let $\pi_n: X \to U_n \backslash X$ be the quotient map. Since $U_n \unlhd U$ and $U$ acts on $X$, both $U$ and $U_n\backslash U$ act on $U_n \backslash X$. 
The Morse function $\bar{h}$ is $\mathbf{SL_3}(\mathbb{Q}[t])$-invariant, and it induces a Morse function on $U_n \backslash X$, which we will also call $\bar h$. 

 Let $z_n$ be the vertex in $X$ which corresponds to the lattice $$t^{2n} e_1 \oplus t^{n} e_2 \oplus e_3$$

\begin{lemma}\label{lemma:ustabz}
The vertex $\pi_n(z_n)$ is stabilized by $U$. 
\end{lemma}
\begin{proof}
Let $u \in U$. Then 
$$u=\left(\begin{array}{ccc}1 & p_x & p_z \\0 & 1 & p_y \\0 & 0 & 1\end{array}\right)$$ where $p_x, p_y, p_z \in \mathbb{Q}[t]$. We will write $u = u_1u_2$ where $u_1 \in U_n$ and $u_2$ stabilizes $z_n$.
Any polynomial $p \in \mathbb{Q}[t]$ can be written as a sum $p = p' + p''$ where $p' \in (t^{n+1})\mathbb{Q}[t]$ and $deg(p'')\leq n$. Write $p_x = p'_x + p''_x$ and $p_y=p'_y+p''_y$. Let $q_z = p_z -p'_xp''_y$ and write $q_z = q_z' + q''_z$.
Note that $u=u_1u_2$, where $$u_1 = \left(\begin{array}{ccc}1 & p'_x & p'_z \\0 & 1 & p'_y \\0 & 0 & 1\end{array}\right)$$ $$u_2=\left(\begin{array}{ccc}1 & p''_x & p''_z \\0 & 1 & q''_y \\0 & 0 & 1\end{array}\right)$$
Note that $u_1 \in U_n$. Since $deg(p''_x), deg(p''_y), deg(q''_z) \leq n$, $u_2$ stabilizes $z_n$ (by Lemma \ref{lemma:stabforms}).  Therefore $$u \pi_n(z_n) = u U_n z_n = U_n u z_n = U_n u_1 u_2 z_n = U_n z_n = \pi_n(z_n)$$
\end{proof}
We will abuse notation slightly and use $z_n$ to denote both the vertex in $X$, and its image $\pi_n(z_n)$ in the quotient. 
\begin{lemma} \label{lemma:bipartite}$Lk^\downarrow(z_n, U_n \backslash X)$ is a complete bipartite graph.
\end{lemma}
\begin{proof}
First, we observe that $Lk^\downarrow(z_n, U_n \backslash X) = U_n \backslash Lk^\downarrow(z_n, X)$.
We have previously shown that $Lk^\downarrow(z_n, X)$ is the orbit of a single 1-cell under elementary matrices $e_{12}(at^n),$ and $e_{23}(bt^n)$, where $a, b \in \mathbb{Q}$. Let $\hat e$ denote the image of this edge in $U_n\backslash X$. In $U_n \backslash U$, $e_{12}(at^n)$ and $e_{23}(bt^n)$ commute, since their commutator is in $U_n$. Suppose $u \in U_n \backslash U$ stabilizes $z_n$. Then there are elements $u_1 = e_{12}(at^n)$ and $u_2 = e_{23}(bt^n)$ such that $u \hat e =   u_1u_2\hat e$. Furthermore, $u_1$ and $u_2$ each fixes exactly one vertex of $\hat e$ and moves the vertex which the other fixes. This gives a labelling of every vertex in $Lk^\downarrow(z_n, U_n \backslash X)$ by a rational number, and every edge by an ordered pair of rational numbers. Since there are no restrictions on $a$ and $b$, all pairs of rational numbers are possible and because of the action on $X$, different ordered pairs of rational numbers give different edges. Hence $Lk^\downarrow(z_n, U_n \backslash X)$ is a complete bipartite graph.
\end{proof}

From this point on, we will let 
\begin{align*}
S^\downarrow_n &= Star^\downarrow(z_n, X)\\
\hat{S}^\downarrow_n &= Star^\downarrow(z_n, U_n\backslash X)\\
L^\downarrow_n &= Lk^\downarrow(z_n, X)\\
\hat{L}^\downarrow_n &= Star^\downarrow(z_n, U_n\backslash X)\\\end{align*}
\subsection{Local cocycles}

\begin{lemma}  There is an infinite family of (nontrivial) $U$-invariant cocycles $\varphi_n \in H^2(\hat{S}^\downarrow_n, \hat{L}^\downarrow_n; \QQ)$. 
\end{lemma}
\begin{proof}
Relative cycles in $H_2(\hat{S}^\downarrow_n, \hat{L}^\downarrow_n; \QQ)$ correspond to cycles in 
$H_1(\hat{L}^\downarrow_n; \QQ)$. 
By Lemma \ref{lemma:bipartite}, $\hat{L}^\downarrow_n$ is a complete bipartite graph, the vertices of each type are parametrized by $\QQ$, and the edges can be labeled by ordered pairs of rational numbers.  (In fact, $\hat{L}^\downarrow_n = \QQ \star \QQ$.) 
We fix an orientation from one family of vertices to the other family, and define a function on the edges $\{\eta_{(q, r)}\}_{ q, r \in \QQ}$ of $\hat{L}^\downarrow_n$ by taking $\varphi_n(\eta_{(q, r)}) = qr$. 

To verify that $\varphi_n$ is a cocycle, note that $\hat{L}^\downarrow_n$ is a graph, so there are no nontrivial 2-coboundaries on $\hat{L}^\downarrow_n$. 

Next, we will show that $\varphi_n$ is $U$-invariant. The loops of length 4 in $\hat{L}^\downarrow_n$ form a generating set for $H_1(\hat{L}^\downarrow_n, \mathbb{Q})$, so it suffices to check that $\varphi_n$ is $U$-invariant on loops of length 4. If $\sigma$ is a loop of length 4, then $\sigma$ has the form $$\eta_{(q_1, r_1)} - \eta_{(q_2, r_1)} + \eta_{(q_2, r_2)} - \eta_{(q_1, r_2)}$$ and $\varphi_n(\sigma) = (q_1-q_2)(r_1-r_2)$. If $u \in U$, then $u$ stabilizes $z_n$ and acts by addition of the degree $n$ coefficient of the $u_{12}$ and $u_{23}$ entries on the coordinates of the subscript, so $$\varphi_n(u\sigma) = ((q_1 + q) - (q_2 + q))((r_1+r) - (r_2 + r)) = \varphi_n(\sigma)$$

$U_{n}$ acts trivially on $U_{n}\backslash U\mathcal{S}_0$, so the value of $\varphi_n$ is invariant under the action of $U$. 

Finally, we will show that $\varphi_n$ is nontrivial by exhibiting a cycle $\hat{\sigma} _n \in H_1(\hat{L}^\downarrow_n; \QQ)$ such that $\varphi_n (\hat\sigma_n) \neq 0$.  Let $\hat\sigma_n = 2 \eta_{(0, 0)} + \eta_{(-1, 0)} + \eta_{(-1,1)} - \eta_{(0, 1)}-\eta_{(0,-1)} + \eta_{(1,-1)} - \eta_{(1,0)}$.  Using the form of $\varphi_n$ given above, we see that $\varphi_n(\hat\sigma_n) = -2$. 
\end{proof}
\begin{lemma} There is a relative $2$-cycle $\sigma_n \in H_2(S^\downarrow_n, L^\downarrow_n; \QQ)$ such that $\pi_n(\sigma_n) = \hat\sigma_n$.
\end{lemma}
\begin{proof}
Let $\mathcal{C}_n$ be the chamber in $St^\downarrow(z_n, \mathcal{A}_0)$, $\bar{\mathcal{C}}_n$ the corresponding edge in $Lk^\downarrow(z_n, \mathcal{A}_0)$, and $$u_1 = \left(\begin{array}{ccc}1 & t^n & 0 \\0 & 1 & 0 \\0 & 0 & 1\end{array}\right), \ \ u_2 = \left(\begin{array}{ccc}1 & 0 & 0 \\0 & 1 & t^n  \\0 & 0 & 1\end{array}\right)$$
Take$$\sigma_n = {\mathcal{C}}_n - u_1^{-1}{\mathcal{C}}_n + u_1^{-1}u_2 {\mathcal{C}}_n - u_1^{-1}u_2u_1 {\mathcal{C}}n + [u_1^{-1},u_2] {\mathcal{C}}_n - u_1u_2^{-1}u_1^{-1}{\mathcal{C}}_n +u_1u_2^{-1}{\mathcal{C}}_n -u_1{\mathcal{C}}_n$$
Since $[u_1^{-1}, u_2] = [u_1, u_2^{-1}]$, $\sigma_n$ is a cycle.  Let $\bar{\sigma}_n$ be the corresponding $1$-cycle in $H_1(L^\downarrow_n;\QQ)$. Note that $u_1$ and $u_2$ descend to nontrivial elements of $U^n$, and their images commute. For each edge $u\bar{\mathcal{C}}_n$ in $\bar{\sigma}_n$, we know that $\pi_n(u\bar{\mathcal{C}}_n) = \eta_{(a, b)}$ for some $a, b \in \mathbb{Q}$. To find $a$, count the number of times $u_1$ appears in $u$ (counting $u_1^{-1}$ as $-1$). To find $b$, count the number of times $u_2$ appears in $u$. For example, $\pi_n(u_1\bar{\mathcal{C}}_n) = \eta_{(1, 0)}$, $\pi_n(u_1u_2^{-1}) = \eta_{(1, -1)}$, and $\pi_n(\bar{\mathcal{C}}_n) = \pi_n([u_1^{-1},u_2]\bar{\mathcal{C}}_n) = \eta_{(0, 0)}$.

Therefore, $\pi_n(\sigma_n) = \hat\sigma_n$.
\end{proof}

We will use $\varphi_n$ to define a cocycle $\Phi_n \in H^2(\Gamma \backslash Y; \QQ)$ by lifting $2$-cells in $\Gamma \backslash Y$ to disks in $X$, applying the quotient map $\pi_n$ to obtain a disk in $U_{n}\backslash X$, evaluating $\varphi_n$ on the intersection with $\hat{L}^\downarrow_n$, and averaging over $\Gamma$-translates of the lifted disk.

\begin{lemma} If $D$ is a $2$-disk in $X$ with boundary in $\psi(Y_0')$, then for sufficiently large $n$, $\pi_n(D)\cap \hat{S}_n \subset \hat S^\downarrow _n$
\end{lemma}
\begin{proof}
To prove the lemma, we will show that if any chamber in $\hat S_n - \hat S_n^\downarrow$ is contained in $\text{Supp}(\pi_n(D)\cap \hat S_n)$, then there exists a geodesic segment $\rho \subset \text{Supp}(\pi_n(D))$ with one endpoint at $z_n$ and the other endpoint at $z \in \partial\pi_n(D)$ with $h(z) > h(z_n)$, which contradicts the fact that $\bar h(\psi(Y_0'))$ is bounded above, since the sequence $\{z_n\}$ has unbounded height. 
There are two chambers in $\hat S_n - \hat S^\downarrow_n$ which have exactly one vertex which is higher than $z_n$.  If $\text{Supp}(\pi_n(D)\cap \hat S_n)$ contains either one of these two chambers, then it must also contain a chamber with two vertices that are higher than $z_n$, because there is a unique chamber adjacent to the ``upper" edge of this chamber. 

Let $\hat{\mathcal{C}}_1$ be the chamber in the support of $\pi_n(D)$ with $h(v) > h(z_n)$ for all vertices $v \neq z_n$. 

There is a face $\mathcal{F}_1$ of $\hat{\mathcal{C}}_{1}$ which is contained in $Lk(z_n, U_n \backslash X)$ such that $h(y) > h(z_n)$ for all $y \in \mathcal{F}_1$.  There is some vertex $v_1$ of $\mathcal{F}_1$ which is in $\pi_n(\mathcal{A}_0)$.
Because $U_n \backslash X$ has no branching along walls of $\pi_n(\mathcal{A}_0)$ which are above $z_n$, the geodesic ray in $\pi_n(\mathcal{A}_0)$ based at $z_n$ and passes through the vertex $v_1$ must eventually intersect $\partial \pi_n(D).$ Call this geodesic ray $\rho$, and notice that $\bar h \circ \rho$ is a strictly increasing function. Therefore, the point where $\rho$ intersects $\partial \pi_n(D)$ is strictly higher (with respect to $\bar h$) than $z_n$, which is a contradiction.
\end{proof}
 
Let $U_\Gamma = U\cap \Gamma$. By Lemma \ref{lemma:ustabz}, $U_\Gamma$ stabilizes $z_n$ and therefore $U_\Gamma \hat{S}^\downarrow_n = \hat{S}^\downarrow_n$ for every $n$.

\begin{lemma}  There is an infinite family of (nontrivial) cocycles $\Phi_n \in H^2(\Gamma \backslash Y; \QQ)$.
\end{lemma}
\begin{proof}
Given $\Gamma B$ a $2$-cell in $\Gamma \backslash Y$, let 

$$\Phi_n(\Gamma B) = \sum_{\gamma V_n \in \Gamma/V_n} \varphi_n(\pi_n(\gamma^{-1}\psi(B)) \cap \hat{S}^\downarrow_n)$$

$\Phi_n$ is well-defined, i.e. the value of $\Phi_n$ is independent of the choices of coset representatives $\gamma U$ and the choice of a lift $B$ for $\Gamma B$:
{First we check that replacing $\gamma$ with $\gamma u_\gamma$ (changing the coset representatives) does not change the value of $\Phi_n$:
\begin{align*}
&\displaystyle\sum_{(\gamma u_\gamma) U_\Gamma \in \Gamma/U_\Gamma}\varphi_n \left(   \pi_n((\gamma u_\gamma)^{-1}\psi(B)) \cap \hat{S}^\downarrow_n\right)\\
&= \displaystyle\sum_{(\gamma u_\gamma) U_\Gamma \in \Gamma/U_\Gamma}\varphi_n \left(   \pi_n(u_\gamma^{-1}\gamma^{-1}\psi(B)) \cap u_\gamma^{-1}\hat{S}^\downarrow_n\right)\\
&= \displaystyle\sum_{\gamma U_\Gamma \in \Gamma/U_\Gamma}\varphi_n \left(  (u_\gamma)^{-1} [ \pi_n(\gamma^{-1}\psi(B)) \cap  \hat{S}^\downarrow_n]\right)\\
&= \displaystyle\sum_{\gamma U_\Gamma \in \Gamma/U_\Gamma}\varphi_n\left(\pi_n(\gamma^{-1}\psi(B)) \cap \hat{S}^\downarrow_n\right) = \Phi_n(\Gamma B)
\end{align*}

Next we check that choosing a different lift of $\Gamma B$ does not change the value of $\Phi_n(\Gamma B)$.  If $y \in \Gamma$, then

\begin{align}
\Phi_n(\Gamma y B) &= \displaystyle\sum_{\gamma U_\Gamma \in \Gamma/U_\Gamma}\varphi_n \left(  \psi_n(\gamma^{-1}y B) \cap  \hat{S}^\downarrow_n\right)\\
&= \displaystyle\sum_{\gamma U_\Gamma \in \Gamma/U_\Gamma}\varphi_n \left(   \pi_n((y^{-1}\gamma)^{-1}\psi(B)) \cap \hat{S}^\downarrow_n\right)\\
&= \displaystyle\sum_{y\gamma U_\Gamma \in \Gamma/U_\Gamma}\varphi_n \left(   \pi_n((y^{-1}y\gamma)^{-1}\psi(B)) \cap \hat{S}^\downarrow_n\right)\\
&= \displaystyle\sum_{y\gamma U_\Gamma \in \Gamma/U_\Gamma}\varphi_n \left(   \pi_n(\gamma^{-1}\psi(B)) \cap \hat{S}^\downarrow_n\right)\\
&= \displaystyle\sum_{\gamma U_\Gamma \in \Gamma/U_\Gamma}\varphi_n \left(   \pi_n(\gamma^{-1}\psi(B)) \cap \hat{S}^\downarrow_n\right)\\
&= \Phi_n(\Gamma B)
\end{align}
}

In order to show that $\Phi_n$ is a cocycle in $H^{2}(\Gamma \backslash Y; \mathbb{Q})$, we will show that it is trivial on boundaries of $3$-disks, and thus is in the kernel of the coboundary map.

Let $\Gamma B^3$ be a $3$-cell in $\Gamma \backslash Y$, corresponding to the $3$-cell $B^3$ in $Y$.  Then $\partial (\Gamma B^3) = \Gamma (\partial B^3)$ is a $2$-sphere in $\Gamma \backslash Y$ and $\partial B^3$ is a $2$-sphere in $Y$.
Since $X$ contains no nontrivial $2$-spheres, the image of $\partial B^3$ under the map $\psi:~Y \to X$ is homotopic to a point.  Thus,
$$\Phi_n(\Gamma (\partial B^n))=\displaystyle\sum_{\gamma U_\Gamma \in \Gamma/U_\Gamma}\varphi_n \left(  \pi_n(\gamma^{-1}\psi(\partial B^3)) \cap \hat{S}^\downarrow_n\right) = 0$$

\end{proof}

\begin{lemma}  \label{lemma:cycle} For each $n$, there is a $2$-cycle $\tilde{\sigma}_n \in H_2(\Gamma \backslash Y; \QQ)$ such that $\Phi_n(\tilde\sigma_n) \neq 0$ and $\Phi_m(\tilde\sigma_n) = 0$ for $m \geq n+1$.
\end{lemma}
This lemma is essentially proved in {\cite{Wortman2013}}. We restate it with minor adaptations of the notation.
\begin{proof}

$\partial \sigma_n$ is a $1$-sphere in $X$ with $h(\partial \sigma_n) < h(z_n)$. Let $v_1, \ldots, v_k$ be the vertices of $\partial \sigma_n$. For $1 \leq i \leq k$, choose a path $p_i:[0,1] \to X$ such that $p_i(0) = v_i$, $p_i(1) \in Y_0$, and $h\circ p_i$ is strictly decreasing. (One choice of $p_i$ would be to choose an efficient simplicial path to $\mathcal{C}_0$ if $v_i \in \mathcal{A}_0$, and an efficient simplicial path to $u\mathcal{C}_0$ if $v_i \in u\mathcal{A}_0$ for $u \in \langle u_1, u_2\rangle$.) Let $e_1, \ldots, e_m$ be the $1$-cells of $\partial \sigma_n$, with $\partial e_i = v_j \cup v_l$. For $1\leq i \leq m$, there is a homotopy relative $p_j(1)$ and $p_l(1)$ between $p_j\cup p_l \cup e_i$ and a path in $Y_0$. This homotopy gives a disk $d_i$, and $\cup_{i=1}^m d_i$ gives a homotopy between $\partial \sigma_n$ and a $1$-sphere $\tilde\sigma_n$ in $Y_0$. Since $Y$ is simply connected, there is a disk $D_n \subset Y$ with $\pi_n\circ\psi(\partial D_n) = \tilde \sigma_n$.  Because filling disks in $X$ are unique, $\pi_n\circ\psi(D_n) \cap \S^\downarrow_n = \sigma_n$. Let $p$ be the quotient map from $Y$ to $\Gamma \backslash Y$. Then $p(D_n)$ is a cycle, because $p(\partial D_n) \subset p(Y_0)$ is trivial. Take $\tilde{\sigma}_n = p(D_n)$. To complete the proof, note that the maximum value of $h$ on $\psi(D_n)$ is attained at $z_n$, and $h(x) > h(z_n)$ if $x \in Lk^{\downarrow}(z_m)$ for $m > n$, so $\psi(D_n) \cap L^{\downarrow}_m = \emptyset$.  
\end{proof}

This final lemma shows that $\{\tilde{\sigma}_n\}$ is an infinite independent family of 2-cycles in $H_2(\Gamma \backslash Y; \QQ)$, and thus completes the proof of Proposition \ref{prop:infdim}.

\section{Proof of the main result}

We now prove Theorem~\ref{theorem:main}.
\begin{proof}
Let $\mathscr{H}_q= \{H_q(\Gamma_\sigma; \QQ)\}$ and
consider the spectral sequence $$E^2_{p,q} = H_p(\Gamma \backslash Y, \mathscr{H}_q).$$

A common reference for this spectral sequence is \cite{Brown1982}. In section VII.8, it is shown that $$E^2_{p,q} \Rightarrow H_{p+q}(\Gamma; C(Y; \QQ))$$ where $C(Y; \QQ)$ is the cellular chain complex of $Y$ with coefficients in $\QQ$.

Because $Y$ is $2$-connected, there is a cellular map $f: Y \to \{\text{pt}\}$ which induces an isomorphism $f_*: H_i(Y; \QQ) \to H_i(\{\text{pt}\}; \QQ)$ for $0 \leq i \leq 2$.  Therefore, $f$ also induces an isomorphism $H_i(\Gamma; C(Y; \QQ)) \to H_i (\Gamma; C(\{\text{pt}\}; \QQ)) \cong H_i(\Gamma; \QQ)$ for $i\leq 2$.

The relevant terms of the spectral sequence are $E^r_{2, 0}$, and  $E^r_{0, 1}$ for $r\geq 2$.
First, we note that
$$E^2_{2, 0} = H_{2}(\Gamma \backslash Y, H_0(\Gamma_\sigma; \QQ))=H_{2}(\Gamma \backslash Y; \QQ).$$
We have demonstrated in Proposition \ref{prop:infdim} that $H_{2}(\Gamma \backslash Y; \QQ)$ is infinite dimensional.

Next, we note that when $q>0$, $$E^2_{p, q} = H_{p}(\Gamma \backslash Y, \{H_q(\Gamma_\sigma; \QQ)\}).$$

%should this be FP_k for all k, or is it okay?
By Proposition \ref{prop:Y}, the cell stabilizers $\Gamma_\sigma$ are of type $F_1$, so $H_1(\Gamma_\sigma; \QQ)$ is finite dimensional for every $0$-cell $\sigma$ in $\Gamma\backslash Y$.

Since $\Gamma$ acts freely on $Y-Y_0$, and the image of $Y_0$ in the quotient consists of a single $2$-dimensional chamber, with finitely many subcells, $H_{0}(\Gamma \backslash Y, \{H_1(\Gamma_\sigma; \QQ)\})$ consists of finite sums in the form $$\sum_{i=0}^N a_i \sigma_i$$ where $a_i \in H_q(\Gamma_\sigma; \QQ)$.  Notice that $H_1(\Gamma_\sigma; \QQ) = 0$ for all but finitely many $\sigma_i$, and is always finite dimensional.  Thus $E^2_{0, 1} = H_{0}(\Gamma \backslash Y, \{H_1(\Gamma_\sigma; \QQ)\})$ is finite dimensional.

To compute $E^r_{2, 0}$ for $r>2$, we note that the kernel of any homomorphism $E^2_{2, 0} \to E^2_{0,1}$ must be infinite dimensional.  Later differentials emanating from $E^r_{2,0}$ are zero, since $E^r_{p,q} = 0$ outside the first quadrant. Thus $E^r_{2, 0}$ is infinite dimensional for all $r$.

In the limit, $H_{n-1}(\Gamma; \QQ)$ is infinite dimensional, and thus $H^{n-1}(\Gamma; \QQ)$ is infinite dimensional as well.

\end{proof}

\bibliographystyle{amsalpha}
\bibliography{general.bib}

\end{document}